\documentclass[12pt]{amsart}
\usepackage[bottom=2cm, right=2cm, left=2cm, top=2cm]{geometry}
\usepackage{amssymb,amsmath,amsthm,mathtools,amsfonts,times,hyperref,mathrsfs,multirow}
\usepackage{pgfplots}
\usepackage{longtable}
\usepackage{booktabs}
\usepackage{caption}
\pgfplotsset{compat=1.15}
\usepackage{mathrsfs,mathdots}
\usetikzlibrary{arrows}
\usepackage{enumerate}
\usepackage{graphicx}
\usepackage{array}
\usepackage{ytableau}
\usepackage{pstricks,pst-node,graphicx}
\usepackage{tikz,varwidth}
\usepackage{setspace}
\usepackage{float}
\usetikzlibrary{calc}
\usepackage[super]{nth}
\usepackage{tikz}
\usepackage{extarrows}
\usepackage{pgfplots}
\usepackage{hyperref}
\usepackage{cleveref}
\pgfplotsset{compat=1.15}
\usetikzlibrary{arrows}

\newtheorem{theorem}{Theorem}[section]

\theoremstyle{definition}

\theoremstyle{remark}
\newtheorem{remark}{Remark}
\numberwithin{equation}{section}

\allowdisplaybreaks

\newcommand\numberthis{\addtocounter{equation}{1}\tag{\theequation}}
\newcommand\restr[2]{{
  \left.\kern-\nulldelimiterspace 
  #1 
  \littletaller 
  \right|_{#2} 
  }}
\newcommand{\littletaller}{\mathchoice{\vphantom{\big|}}{}{}{}}  

\begin{document}
\title[Separable integer partition classes and Slater's list - I]{Separable integer partition classes and Slater's list - I}
\author{Aritram Dhar}
\address{Department of Mathematics, University of Florida, Gainesville,
FL 32611, USA}
\email{aritramdhar@ufl.edu}
\author{Ankush Goswami}
\address{Department of Mathematics, University of Texas Rio Grande Valley, Edinburg, TX 78541, USA}
\email{ankush.goswami@utrgv.edu}
\author{Runqiao Li}
\address{Department of Mathematics, University of Texas Rio Grande Valley, Edinburg, TX 78541, USA}
\email{runqiao.li@utrgv.edu}

\date{\today}



\begin{abstract}
Slater’s list of Rogers--Ramanujan type identities consists of 130 series--product identities whose analytic proofs rely primarily on Bailey pair techniques. Although these identities play an important role in the theory of $q$-series and partitions, combinatorial interpretations for many of them remain unknown, largely because the series sides are difficult to interpret naturally in terms of partitions. In this paper we apply Andrews’ theory of separable integer partition (SIP) classes to several identities from Slater’s list. By constructing suitable SIP classes, we obtain natural partition-theoretic interpretations and parameterized generalizations of their series sides. We then apply various $q$-hypergeometric transformations to these generalized series to derive alternative expressions, which in certain cases reduce to infinite products. These results illustrate how the SIP framework provides a systematic approach to understanding Rogers--Ramanujan type identities and offer new combinatorial insights into identities appearing in Slater’s list.

\end{abstract}

\maketitle
\section{Introduction}

A \textit{partition} $\lambda$ is a non-increasing finite sequence $\lambda = (\lambda_1,\lambda_2,\dots,\lambda_k)$ of positive integers. The elements $\lambda_i$ appearing in the sequence $\lambda$ are called the \textit{parts} of $\lambda$. The number of parts or \textit{length} of $\lambda$ is denoted by $\#(\lambda)$ and the largest part of $\lambda$ is denoted by $l(\lambda)$. The sum of all the parts of $\lambda$ is called the \textit{size} of $\lambda$ and denoted by $|\lambda|$. We say $\lambda$ is a partition of $n$ if its size is equal to $n$.

For complex variables $a$, $a_1,a_2,\ldots,a_k$, $q$, we define the conventional $q$-Pochhammer symbol as
\begin{align*}
(a)_n = (a;q)_n := \prod\limits_{i=0}^{n-1}(1-aq^i)\quad\text{and}\quad(a_1,a_2,\ldots,a_k;q)_n := \prod\limits_{i=1}^{k}(a_i)_n,
\end{align*}
where $n\in\mathbb{Z}_{\ge 0}\cup\{\infty\}$, and $|q|<1$ here and throughout the paper.

Slater’s celebrated list of Rogers--Ramanujan type identities \cite{Slater1952} occupies a central position in the theory of $q$-series and integer partitions. Compiled in the early 1950s, it consists of 130 series--product identities with deep connections to partition theory, modular forms, and basic hypergeometric series. Slater’s original proofs relied on Bailey pair techniques and, as a consequence, do not provide combinatorial interpretations of the identities. Since then, there have been sporadic efforts to interpret certain identities from Slater’s list combinatorially. However, these interpretations are typically tailored to individual identities and do not extend in a systematic way to other identities in the list. A central difficulty in such attempts lies in giving a natural combinatorial interpretation of the series side of these identities. 

A major advance in this direction was made by Andrews, who introduced the framework of \emph{Separable Integer Partition (SIP)} classes \cite[Section $2$]{AndrewsSIP}. This framework provides a structural approach for obtaining combinatorial interpretations of various $q$-series by decomposing partitions into a finite ``basis'' part together with a freely extendable tail consisting of multiples of a fixed modulus.

We consider the first G\"{o}llnitz--Gordon identity, which also appears as identity (36) in Slater's list.
\begin{equation}\label{GG1}
\sum_{n\ge0}\frac{q^{n^2}(-q;q^2)_n}{(q^2;q^2)_n}
=
\frac{1}{(q,q^4,q^7;q^8)_\infty}.
\end{equation}
It is immediate that the product side generates partitions into parts congruent to $\pm 1$ or $4$ modulo $8$. The combinatorial interpretation of the series side, however, is more subtle. G\"ollnitz, and independently Gordon, showed that it enumerates partitions in which the difference between parts is at least $2$, with the stronger condition that the difference between even parts is at least $4$.

In \cite{AndrewsSIP}, Andrews showed that the series side of \eqref{GG1} arises naturally as the generating function of a separable integer partition (SIP) class. Moreover, in the same work he demonstrated that several other classical partition identities, including Schur's 1926 theorem, admit interpretations within the SIP framework. Thus, SIP classes provide a unified combinatorial mechanism for interpreting a wide range of $q$-series identities.

The aim of this paper is twofold. First, we use the SIP framework to obtain generalizations of the series sides of several identities from Slater's list (see Table \ref{tab:slater-identities} below). Second, we apply $q$-hypergeometric transformations to these series to derive alternative expressions, which in certain special cases yield infinite products. Our work provides a starting point for a broader combinatorial understanding of Rogers--Ramanujan type identities using the SIP framework. 
\renewcommand{\arraystretch}{3}
\setlength{\tabcolsep}{8pt}
\setlength{\aboverulesep}{0pt}
\setlength{\belowrulesep}{0pt}
{\small\begin{longtable}{>{\raggedright\arraybackslash}p{11.5cm}
>{\centering\arraybackslash}p{1.5cm}
>{\centering\arraybackslash}p{2.5cm}}

\toprule

\textbf{Identity} & \textbf{Equation} & \textbf{Section}\\

\midrule
\endfirsthead

\toprule
\textbf{Identity} & \textbf{Equation} & \textbf{Section} \\
\midrule
\endhead
\endfoot

\bottomrule
\caption{Identities from Slater's list considered in this paper, together with their equation numbers in the list and the sections in which their generalizations are developed.}
\label{tab:slater-identities}
\endlastfoot

$\displaystyle
\sum_{n=0}^{\infty}\frac{(-1)^{n}(-q;q^2)_{n}q^{n^2}}{(q^4;q^4)_{n}}
=(q;q^2)_{\infty}(q^2;q^4)_{\infty}
$
& (4) & Section~\ref{sec:SlaterList4}\\

$\displaystyle
\sum_{n=0}^{\infty}\frac{(-q;q^2)_{n}q^{n^2}}{(q^4;q^4)_{n}}
=\dfrac{(q^3;q^6)^2_\infty(q^6;q^6)_\infty(-q;q^2)_\infty}{(q^2;q^2)_\infty}
$
& (25) & Section~\ref{sec:SlaterList4} \\

$\displaystyle
\sum_{n=0}^{\infty}\frac{(q;q^2)_{2n}q^{4n^2}}{(q^4;q^4)_{2n}}
=\dfrac{(q^5,q^7,q^{12};q^{12})_\infty}{(q^4;q^4)_\infty}
$
& (51) & Section~\ref{sec:SlaterList4} \\

$\displaystyle
\sum_{n=0}^{\infty}\frac{(q;q^2)_{2n+1}q^{4n(n+1)}}{(q^4;q^4)_{2n+1}}
=\dfrac{(q,q^{11},q^{12};q^{12})_\infty}{(q^4;q^4)_\infty}
$
& (55) & Section~\ref{sec:SlaterList4}\\

$\displaystyle
\frac{(q;q)_{\infty}}{(-q;q)_{\infty}}
\sum_{n=0}^{\infty}\frac{(-q;q)_{n}q^{\frac{n(n+1)}{2}}}{(q;q)_{n}}
=(q,q^3,q^4;q^4)_{\infty}
$
& (8) & Section~\ref{sec:SlaterList5} \\
$\displaystyle
\sum_{n=0}^{\infty}\frac{(-q;q)_{n}}{(q;q)_{n}}q^{\frac{n(n-1)}{2}}=\frac{(-q;q)_{\infty}}{(q;q)_{\infty}}\left((q,q^3,q^4;q^4)_{\infty}+(q^2,q^2,q^4;q^4)_{\infty}\right)$ & (13) & Section~\ref{sec:SlaterList5}\\

$\displaystyle
\sum_{n=0}^{\infty}\frac{(-1)^nq^{n(3n-2)}}{(q^4;q^4)_{n}(-q;q^2)_{n}}
=\frac{(q,q^4,q^5;q^5)_{\infty}}{(q^2;q^2)_{\infty}}
$
& (15) & Section~\ref{sec:SlaterList6}\\

$\displaystyle
\sum_{n=0}^{\infty}\frac{(-1)^nq^{3n^2}}{(q^4;q^4)_{n}(-q;q^2)_{n}}
=\frac{(q^2,q^3,q^5;q^5)_{\infty}}{(q^2;q^2)_{\infty}}
$
& (19) & Section~\ref{sec:SlaterList6} \\

$\displaystyle
\sum_{n=0}^{\infty}\frac{(-1)^{n}q^{n(2n+1)}}{(q^2;q^2)_{n}(-q;q^{2})_{n+1}}
=(q;q^2)_{\infty}(-q^2;q^2)_{\infty}
$
& (5) & Section~\ref{sec:SlaterList7} \\
\end{longtable}}

The rest of the paper is organized as follows. In Section~\ref{Sec:preliminary}, we review the necessary background on SIP classes and present the identities used throughout the paper. In Sections~\ref{sec:SlaterList4}, \ref{sec:SlaterList5}, and \ref{sec:SlaterList6}, we study the identities listed above, grouping them according to similarities in their partition interpretations. Finally, in Section~\ref{sec:SlaterList7}, we discuss the results from a combinatorial perspective and suggest directions for further investigation of Slater's list.

\section{Preliminary and Background}\label{Sec:preliminary}
We start with the definition of an SIP class. Let $P$ be a class of integer partitions and let $k$ be a fixed positive integer. The class $P$ is called a \emph{separable integer partition class} (SIP class) with modulus $k$ if there exists a subset $B\subset P$, called the \emph{basis}, satisfying the following properties:

\begin{enumerate}
\item For every $n\ge1$, the number of partitions in $B$ with $n$ parts is finite.
\item Every partition $\lambda=(\lambda_1,\dots,\lambda_n)\in P$ can be written uniquely in the form
\[
\lambda_i=b_i+\pi_i \qquad (1\le i\le n),
\]
where $(b_1,\dots,b_n)\in B$ and $(\pi_1,\dots,\pi_n)$ is a partition into $n$ nonnegative parts each divisible by $k$.
\item Conversely, every partition obtained in this manner belongs to $P$.
\end{enumerate}
For SIP classes, their generating function can be obtained as follows.
\begin{theorem}[Andrews~\cite{AndrewsSIP}]
If $P$ is an SIP class with basis $B$ and  modulus $k$,
$$\sum_{\lambda\in{P}}q^{|\lambda|}=\sum_{n=0}^{\infty}\frac{B(n;q)}{(q^k;q^k)_n},$$
where $B(n;q)$ is the generating function for partitions in $B$ with length $n$.
\end{theorem}
SIP classes can be better understood through Young diagram. The \textit{Young diagram} of a partition $\lambda$ is a way of representing $\lambda$ graphically where the parts $\lambda_i$ of $\lambda$, called \textit{cells}, are arranged in left-justified rows with $\lambda_i$ cells in the $i$-th row. For example, the Young diagram of the partition $(5,4,4,3)$ is
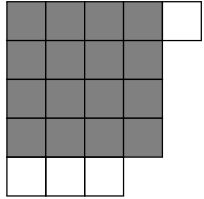
\ytableausetup{boxsize=0.5cm}
\begin{center}
\begin{ytableau}
*(gray) & *(gray) & *(gray) & *(gray) &  \\
*(gray) & *(gray) & *(gray) & *(gray)\\
*(gray) & *(gray) & *(gray) & *(gray)\\
*(gray) & *(gray) & *(gray) & *(gray)\\
 &  & \\
\end{ytableau}
\end{center}
There is a unique largest square that can be fit into the Young diagram of a partition where one corner of the square aligns with the top left corner of the diagram. This unique square is called the Durfee square. In the above example, the Durfee square is of size $4\times 4$ and shaded gray in the Young diagram.

We now use the partition interpretation of \eqref{GG1} as an example to demonstrate the construction of an SIP class. Let $\mathcal{A}$ be the set of partitions $\lambda=(\lambda_1,\lambda_2,\ldots,\lambda_n)$ such that
$$\lambda_i-\lambda_{i+1}\geq2\quad\text{and}\quad\lambda_{i}-\lambda_{i+1}\geq3\quad\text{if $\lambda_i$ is even.}$$
It was shown by Andrews~\cite{AndrewsSIP} that $\mathcal{A}$ is a SIP class of modulus $2$. Giving a partition in $\mathcal{A}$ and its Young diagram, we subtract columns of width $2$ till the remained parts reach the minimal gap condition.

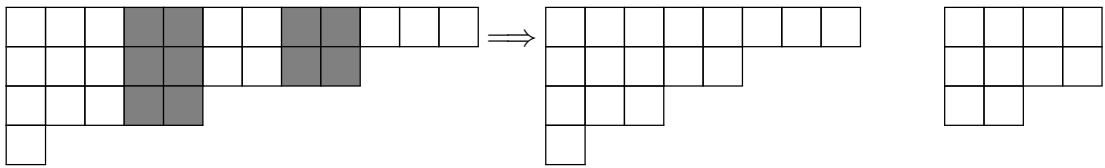
\begin{center}
\ytableausetup{mathmode,boxframe=normal,boxsize=1.2em}
\begin{ytableau}
 \empty  & \empty & \empty & *(gray)\empty & *(gray)\empty & \empty & \empty & *(gray)\empty & *(gray)\empty & \empty & \empty & \empty \\
 \empty & \empty & \empty & *(gray)\empty & *(gray)\empty & \empty & \empty & *(gray)\empty & *(gray)\empty \\
 \empty & \empty & \empty & *(gray)\empty & *(gray)\empty\\
\empty \\
\end{ytableau}
$\Longrightarrow$
\ydiagram{8,5,3,1}
\hspace{1cm}
\ydiagram{4,4,2}
\end{center}
Let $\beta$ be the remained partition and $\mu$ be the partition consists of those subtracted columns. We can see that the $\beta$ satisfies
$$
3\geq\beta_i-\beta_{i+1}\geq2\quad\text{if $\beta_i$ is odd},\quad
4\geq\beta_i-\beta_{i+1}\geq3\quad\text{if $\beta_i$ is even},
$$
while $\mu$ is a partition into non-negative multiples of $2$. This decomposition is unique for each $\lambda\in\mathcal{A}$, and conversely, giving such a pair of $\beta$ and $\mu$, we can always recover a partition $\lambda\in\mathcal{A}$ 
by taking $\lambda_i=\beta_i+\mu_i$. Thus, we have shown that $\mathcal{A}$ is an SIP class of modulus $2$, and the partitions share the same conditions as $\beta$ form the basis. All the SIP classes in the sequel are constructed and can be verified in the same manner. 

We shall also need the $q$-binomial (or Gaussian) coefficients, which are defined as
\begin{align*}
{a+b\brack b}_q    
:= \begin{cases}
\dfrac{(q)_{a+b}}{(q)_a(q)_b},&\text{if}\, a,b\ge 0,\\\mbox{}\\
0,&\text{otherwise}.\end{cases}    
\end{align*}
It is well-known that the $q$-binomial coefficient ${a+b\brack b}_q$ is the generating function for partitions into at most $b$ parts each of size at most $a$ (see \cite[Chapter $3$]{A}), and for integers $0\le m\le n$, we have the following recurrence relations of $q$-binomial coefficients.
\begin{align*}
{n\brack m}_q &= {n-1\brack m}_q + q^{n-m}{n-1\brack m-1}_q
\end{align*}
and
\begin{align*}
{n\brack m}_q &= {n-1\brack m-1}_q + q^m{n-1\brack m}_q.
\end{align*}
We now give a completed list of identities we need for the rest of the paper. We begin by stating Jacobi's triple product identity \cite[p. $21$, $(2.2.10)$]{A}.
\begin{align*}
\sum\limits_{n=-\infty}^{\infty}z^nq^{n^2} = (q^2,-zq,-z^{-1}q;q^2)_{\infty}.\numberthis\label{eq21}
\end{align*}
We then state a finite version of $q$-binomial theorem (replacing $z$ by $-z$ and $N$ by $n$ in \cite[p. $36$, $(3.3.6)$]{A}) which is as follows.
\begin{align*}
(-z;q)_{n}=\sum_{k=0}^{n}z^{k}q^{\binom{k}{2}}{n\brack k}_{q}.\numberthis\label{eq22}
\end{align*}
We next state Watson's $q$-analogue of Whipple's theorem \cite[p. $81$, $(4.1.3)$]{AndrewsBerndt2009}. If $\alpha$, $\beta$, $\gamma$, $\delta$, and $\epsilon$ are complex numbers such that $\beta\gamma\delta\epsilon\neq 0$, and if $N$ is any non-negative integer, then
\begin{align*}
\sum\limits_{n=0}^{\infty}&\dfrac{(\alpha,q\sqrt{\alpha},-q\sqrt{\alpha},\beta,\gamma,\delta,\epsilon,q^{-N};q)_n}{(q,\sqrt{\alpha},-\sqrt{\alpha},\frac{\alpha q}{\beta},\frac{\alpha q}{\gamma},\frac{\alpha q}{\delta},\frac{\alpha q}{\epsilon},\alpha q^{N+1};q)_n}\left(\frac{\alpha^2q^{N+2}}{\beta\gamma\delta\epsilon}\right)^n\\
&= \dfrac{(\alpha q,\frac{\alpha q}{\delta\epsilon};q)_N}{(\frac{\alpha q}{\delta},\frac{\alpha q}{\epsilon};q)_N}\sum\limits_{n=0}^{\infty}\dfrac{(\frac{\alpha q}{\beta\gamma},\delta,\epsilon,q^{-N};q)_n}{(q,\frac{\alpha q}{\beta},\frac{\alpha q}{\gamma},\frac{\delta\epsilon q^{-N}}{\alpha};q)_n}q^n.\numberthis\label{eq23}
\end{align*}
We now state an identity from Ramanujan's Lost Notebook Part I \cite[p. $233$, $(9.4.3)$]{AndrewsBerndt2005}.
\begin{align*}
\sum\limits_{n=0}^{\infty}\dfrac{(-1)^nq^{n(n+1)/2}}{(-q;q)_n} = \sum\limits_{n=0}^{\infty}q^{n(3n+1)/2}(1-q^{2n+1}).\numberthis\label{eq24}    
\end{align*}
Next, we state a result from Fine's book \cite[p. $24$, $(20.21)$]{Fine88}.
\begin{align*}
\sum\limits_{n=0}^{\infty}\dfrac{(qt)^n}{(q,bq;q)_n} = \dfrac{1}{(bq,tq;q)_{\infty}}\sum\limits_{r=0}^{\infty}(-b)^r\dfrac{(tq;q)_r}{(q;q)_r}q^{\frac{r^2+r}{2}}.\numberthis\label{eq25}   
\end{align*}
Lastly, we state a transformation due to Sears from Gasper and Rahman's classic text on basic hypergeometric series \cite[p. $71$, $(3.2.5)$]{GasperRahman2004}.
\begin{align*}
\sum\limits_{k=0}^{\infty}\dfrac{(q^{-n},a,b;q)_k}{(q,d,e;q)_n}\left(\frac{deq^n}{ab}\right)^k = \dfrac{(\frac{e}{a};q)_n}{(e;q)_n}\sum\limits_{k=0}^{\infty}\dfrac{(q^{-n},a,\frac{d}{b};q)_k}{(q,d,\frac{aq^{1-n}}{e};q)_k}q^k.\numberthis\label{eq26}\\
\end{align*}

\section{Partitions with restricted gap mod $4$}\label{sec:SlaterList4}
In Slater's list, Equation $(4)$, Equation $(25)$, Equation $(51)$, and Equation $(55)$ state that
\begin{equation}\label{eq:SlaterList4}
\sum_{n=0}^{\infty}\frac{(-1)^{n}(-q;q^2)_{n}q^{n^2}}{(q^4;q^4)_{n}}=(q;q^2)_{\infty}(q^2;q^4)_{\infty},    
\end{equation}
\begin{equation}\label{eq:SlaterList25}
\sum_{n=0}^{\infty}\frac{(-q;q^2)_{n}q^{n^2}}{(q^4;q^4)_{n}}=\dfrac{(q^3;q^6)^2_\infty(q^6;q^6)_\infty(-q;q^2)_\infty}{(q^2;q^2)_\infty}, 
\end{equation}
\begin{equation}\label{eq:SlaterList53}
\sum_{n=0}^{\infty}\frac{(q;q^2)_{2n}q^{4n^2}}{(q^4;q^4)_{2n}}=\dfrac{(q^5,q^7,q^{12};q^{12})_\infty}{(q^4;q^4)_\infty},
\end{equation}
\begin{equation}\label{eq:SlaterList55}
\sum_{n=0}^{\infty}\frac{(q;q^2)_{2n+1}q^{4n(n+1)}}{(q^4;q^4)_{2n+1}}=\dfrac{(q,q^{11},q^{12};q^{12})_\infty}{(q^4;q^4)_\infty}.  
\end{equation}
We first introduce the partitions associated with the sum sides of these four identities. Let $\mathcal{P}_4$ be the set of strict partitions satisfying the following gap conditions.
\begin{enumerate}
    \item $\lambda_i-\lambda_{i+1}\equiv2\mod4$ if they are both odd;
    \item $\lambda_i-\lambda_{i+1}\equiv0\mod4$ if they are both even;
    \item $\lambda_{i}-\lambda_{i+1}\equiv3\mod4$ if they have different parity;
    \item $\lambda_{i}\equiv1\ \text{or}\ 2\mod4$ if it is the smallest part.
\end{enumerate}
Let $o(\lambda)$ and $e(\lambda)$ be the number of odd parts and even parts in $\lambda$, respectively. We will first show that
\begin{equation}
\sum_{\lambda\in\mathcal{P}_4}x^{o(\lambda)}y^{e(\lambda)}q^{|\lambda|}=\sum_{n=0}^{\infty}\frac{(-yq/x;q^2)_{n}x^{n}q^{n^2}}{(q^4;q^4)_{n}}.
\end{equation}
So, the sum sides of~\eqref{eq:SlaterList4} and~\eqref{eq:SlaterList25} are the special cases of $x=y=-1$ and $x=y=1$, respectively. Meanwhile, Let $\mathcal{P}_{4,e}$ and $\mathcal{P}_{4,o}$ be the set of partitions in $\mathcal{P}_{4}$ with even and odd length, respectively. We will also show that
\begin{equation}\label{eq:P4eGenPre}
\sum_{\lambda\in\mathcal{P}_{4,e}}x^{o(\lambda)}y^{e(\lambda)}q^{|\lambda|}=\sum_{n=0}^{\infty}\frac{(-yq/x;q^2)_{2n}x^{2n}q^{4n^2}}{(q^4;q^4)_{2n}},
\end{equation}
\begin{equation}\label{eq:P4oGenPre}
\sum_{\lambda\in\mathcal{P}_{4,o}}x^{o(\lambda)}y^{e(\lambda)}q^{|\lambda|}=\sum_{n=0}^{\infty}\frac{(-yq/x;q^2)_{2n+1}x^{2n+1}q^{4n^2+4n+1}}{(q^4;q^4)_{2n+1}}.
\end{equation}
So, the sum side of \ref{eq:SlaterList53} is the special case of \ref{eq:P4eGenPre} with $x=y=-1$, while the sum side of \ref{eq:SlaterList55} is the special case of \ref{eq:P4oGenPre} with $x=y=-1$ and then multiplied by $q^{-1}$.
It's clear that $\mathcal{P}_4$ is an SIP class with modulus $4$. Let $\mathcal{B}_4$ be the basis, then it consists of partitions such that
\begin{enumerate}
    \item $\lambda_i-\lambda_{i+1}=2$ if they are both odd;
    \item $\lambda_i-\lambda_{i+1}=4$ if they are both even;
    \item $\lambda_{i}-\lambda_{i+1}=3$ if they have different parity;
    \item $\lambda_{i}=1\ \text{or}\ 2$ if it is the smallest part.
\end{enumerate}
Let $B_{4}(n,h):=B_4(n,h;x,y,q)$ be the generating function of partitions in $\mathcal{B}_4$ with length $n$ and largest part $h$, where $x$ and $y$ keep track of the number of odd parts and even parts, respectively.
\begin{theorem}
The initial values of $B_4(n,h)$ are
\begin{equation}\label{eq:InitialB4}
B_{4}(1,h)=\begin{cases}
xq,     &  \text{if $h=1$},\\
yq^2,   &  \text{if $h=2$},\\
0,      &  \text{otherwise}.
\end{cases}
\end{equation}
\end{theorem}

\begin{theorem}
For any integers $n\geq1$ and $h\geq1$,
\begin{equation}\label{eq:B4EvenRec}
B_4(n,2h)=\frac{yq}{x}B_4(n,2h-1),    
\end{equation}
\begin{equation}\label{eq:B4OddRec}
B_4(n,2h-1)=xq^{2h-1}B_4(n-1,2h-3)+yq^{2h}B_4(n-1,2h-5).
\end{equation}
\end{theorem}

\begin{proof}
We start with \eqref{eq:B4EvenRec}. Given a partition $\lambda=(\lambda_1,\lambda_2,\ldots,\lambda_{\#(\lambda)})$ in $\mathcal{B}_4$ with length $n$ and $\lambda_1=2h$. Then, $\lambda_2$ can be either $2h-4$ or $2h-3$. If we subtract $1$ from $\lambda_1$, the resulted partition will have largest part being odd, and the gap between $\lambda_1-1$ and $\lambda_2$ also satisfies the desired conditions. Thus, we end up with a partition in $\mathcal{B}_4$ with length $n$ and largest part $2h-1$. Meanwhile, by changing the largest part from $2h$ to $2h-1$, the number of odd parts is increased by $1$, and the number of even parts is decreased by $1$, thus we need the factor $yq/x$.

For \eqref{eq:B4OddRec}, since the largest part is $2h-1$, then the second largest part can be either $2h-3$ or $2h-4$. This shows that
$$B_4(n,2h-1)=xq^{2h-1}B_4(n-1,2h-3)+xq^{2h}B_4(n-1,2h-4).$$
Applying \eqref{eq:B4EvenRec} to rewrite $B_4(n-1,2h-4)$, we finish the proof.
\end{proof}

\begin{theorem}
For any integers $n\geq1$ and $h\geq1$,
\begin{equation}\label{eq:Bnh1}
B_{4}(n,2n+2h-1)=x^{n-h}y^{h}q^{n^2+h^2+2h}{n-1\brack h}_{q^2},
\end{equation}
\begin{equation}\label{eq:Bnh2}
B_{4}(n,2n+2h)=x^{n-h-1}y^{h+1}q^{n^2+h^2+2h+1}{n-1\brack h}_{q^2}.
\end{equation}
\end{theorem}

\begin{proof}
We first prove \eqref{eq:Bnh1}. Let the right-hand side of \eqref{eq:Bnh1} be $R_{4}(n,2n+2h-1)$, note that
\begin{align*}
R_4(n,2n+2h-1)=&x^{n-h}y^{h}q^{n^2+h^2+2h}{n-1\brack h}_{q^2}\\
=&x^{n-h}y^{h}q^{n^2+h^2+2h}\left(q^{2h}{n-2\brack h}_{q^2}+{n-2\brack h-1}_{q^2}\right)\\
=&x^{n-h}y^{h}q^{n^2+h^2+4h}{n-2\brack h}_{q^2}+x^{n-h}y^{h}q^{n^2+h^2+2h}{n-2\brack h-1}_{q^2}\\
=&xq^{2n+2h-1}R_4(n-1,2n+2h-3)+yq^{2n+2h}R_4(n-1,2n+2h-5),
\end{align*}
so~\eqref{eq:Bnh1} satisfies the recurrence \eqref{eq:B4OddRec}. By checking the initial value, we see that \eqref{eq:Bnh1} is the desired solution. By \eqref{eq:B4EvenRec}, \eqref{eq:Bnh2} follows immediately. So we finish the proof.
\end{proof}

\begin{theorem}\label{thm:P4Gen}
The generating function of $\mathcal{P}_{4,e}$, $\mathcal{P}_{4,o}$, and $\mathcal{P}_4$ are given by
\begin{equation}\label{eq:P4eGen}
\sum_{\lambda\in\mathcal{P}_{4,e}}x^{o(\lambda)}y^{e(\lambda)}q^{|\lambda|}=\sum_{n=0}^{\infty}\frac{(-yq/x;q^2)_{2n}x^{2n}q^{4n^2}}{(q^4;q^4)_{2n}},
\end{equation}
\begin{equation}\label{eq:P4oGen}
\sum_{\lambda\in\mathcal{P}_{4,o}}x^{o(\lambda)}y^{e(\lambda)}q^{|\lambda|}=\sum_{n=0}^{\infty}\frac{(-yq/x;q^2)_{2n+1}x^{2n+1}q^{4n^2+4n+1}}{(q^4;q^4)_{2n+1}},
\end{equation}
\begin{equation}\label{eq:P4gen}
\sum_{\lambda\in\mathcal{P}_4}x^{o(\lambda)}y^{e(\lambda)}q^{|\lambda|}=\sum_{n=0}^{\infty}\frac{(-yq/x;q^2)_{n}x^{n}q^{n^2}}{(q^4;q^4)_{n}}.    
\end{equation}
\end{theorem}

\begin{proof}
Since $\mathcal{P}_4$ is an SIP class with modulus $4$, so are $\mathcal{P}_{4,e}$ and $\mathcal{P}_{4,o}$. Note that
\begin{align*}\label{eq}
\sum_{\lambda\in\mathcal{P}_4}x^{o(\lambda)}y^{e(\lambda)}q^{|\lambda|}=&1+\sum_{n=1}^{\infty}\sum_{h=0}^{\infty}\frac{B_4(2n,4n+2h-1)+B_4(2n,4n+2h)}{(q^4;q^4)_{2n}}\\
=&1+\sum_{n=1}^{\infty}\frac{x^{2n}q^{4n^2}}{(q^4;q^4)_{2n}}\sum_{h=0}^{\infty}\left((y/x)^{h}q^{h^2+2h}{2n-1\brack h}+(y/x)^{h+1}q^{h^2+2h+1}{2n-1\brack h}\right)\\
=&1+\sum_{n=1}^{\infty}\frac{x^{2n}q^{4n^2}(1+yq/x)}{(q^4;q^4)_{2n}}\sum_{h=0}^{\infty}(y/x)^hq^{h^2+2h}{2n-1\brack h}\\
=&1+\sum_{n=1}^{\infty}\frac{x^{2n}q^{4n^2}(1+yq/x)(-yq^3/x;q^2)_{2n-1}}{(q^4;q^4)_{2n}}\\
=&\sum_{n=0}^{\infty}\frac{x^{2n}q^{4n^2}(-yq/x;q^2)_{2n}}{(q^4;q^4)_{2n}},
\end{align*}
\begin{align*}
\sum_{\lambda\in\mathcal{P}_{4,o}}x^{o(\lambda)}y^{e(\lambda)}q^{|\lambda|}=&\sum_{n=0}^{\infty}\sum_{h=0}^{\infty}\frac{B_4(2n+1,4n+2h+1)+B_4(2n,4n+2h+2)}{(q^4;q^4)_{2n+1}}\\
=&\sum_{n=0}^{\infty}\frac{x^{2n+1}q^{4n^2+4n+1}}{(q^4;q^4)_{2n+1}}\sum_{h=0}^{\infty}\left((y/x)^{h}q^{h^2+2h}{2n\brack h}+(y/x)^{h+1}q^{h^2+2h+1}{2n\brack h}\right)\\
=&\sum_{n=0}^{\infty}\frac{x^{2n+1}q^{4n^2+4n+1}(1+yq/x)}{(q^4;q^4)_{2n+1}}\sum_{h=0}^{\infty}(y/x)^hq^{h^2+2h}{2n\brack h}\\
=&\sum_{n=0}^{\infty}\frac{x^{2n+1}q^{4n^2+4n+1}(1+yq/x)(-yq^3/x;q^2)_{2n}}{(q^4;q^4)_{n}}\\
=&\sum_{n=0}^{\infty}\frac{x^{2n+1}q^{4n^2+4n+1}(-yq/x;q^2)_{2n+1}}{(q^4;q^4)_{2n+1}}.
\end{align*}
So, we conclude that
\begin{align*}
\sum_{\lambda\in\mathcal{P}_4}x^{o(\lambda)}y^{e(\lambda)}q^{|\lambda|}=&\sum_{\lambda\in\mathcal{P}_{4,e}}x^{o(\lambda)}y^{e(\lambda)}q^{|\lambda|}+\sum_{\lambda\in\mathcal{P}_{4,o}}x^{o(\lambda)}y^{e(\lambda)}q^{|\lambda|}\\
=&\sum_{n=0}^{\infty}\frac{x^{2n}q^{4n^2}(-yq/x;q^2)_{2n}}{(q^4;q^4)_{2n}}+\sum_{n=0}^{\infty}\frac{x^{2n+1}q^{4n^2+4n+1}(-yq/x;q^2)_{2n+1}}{(q^4;q^4)_{2n+1}}\\
=&\sum_{n=0}^{\infty}\frac{x^nq^{n^2}(-yq/x;q^2)_{n}}{(q^4;q^4)_{n}}.
\end{align*}
We finish the proof. 
\end{proof}

Observe that the RHS of \eqref{eq:P4gen} is the same when $(x,q)\mapsto (-x,-q)$. 
Next, we substitute $(q,\alpha,\beta,\delta)\rightarrow (q^2,y,-y,-yq/x)$ and let $\gamma,\epsilon,N\longrightarrow\infty$ in \eqref{eq23} to obtain 
\begin{equation}\label{eq317}
\sum_{\lambda\in\mathcal{P}_4}x^{o(\lambda)}y^{e(\lambda)}q^{|\lambda|}=\frac{(-xq;q^2)_{\infty}}{(yq^2;q^2)_{\infty}}\left(1+(1+y)\sum\limits_{n=1}^{\infty}\frac{(-1)^n(1-yq^{4n})(y^2q^4;q^4)_{n-1}(-yq/x;q^2)_nx^nq^{3n^2}}{(q^4;q^4)_n(-xq;q^2)_n}\right).
\end{equation}
Now we are ready to present the sum-product identities related to $\mathcal{P}_4$.
\begin{theorem}\label{thm36}
\begin{equation}\label{eq319}
\sum_{\lambda\in\mathcal{P}_4}x^{o(\lambda)}(-1)^{e(\lambda)}q^{|\lambda|}=\sum_{n=0}^{\infty}\frac{(q/x;q^2)_{n}x^nq^{n^{2}}}{(q^4;q^4)_{n}}=\frac{(-xq;q^2)_{\infty}}{(-q^2;q^2)_{\infty}}.   
\end{equation}
\end{theorem}
\begin{proof}
Theorem \ref{thm36} follows substituting $y = -1$ in \eqref{eq317} above.    
\end{proof}

Let $x\to-1$, we get \eqref{eq:SlaterList4} immediately. Hence, this provides an extension for Equation $(4)$ in Slater's list. 

\begin{theorem}\label{thm35}
\begin{equation}\label{eq:P4x1y-1}
\sum_{\lambda\in\mathcal{P}_4}(-1)^{e(\lambda)}q^{|\lambda|}=\sum_{n=0}^{\infty}\frac{(q;q^2)_{n}q^{n^{2}}}{(q^4;q^4)_{n}}=\frac{(q^2;q^4)_{\infty}^2}{(q;q^2)_{\infty}}.   
\end{equation}
\end{theorem}
\begin{proof}Theorem \ref{thm35} follows substituting $(x,y) = (1,-1)$ in \eqref{eq317} above. Theorem \ref{thm35} also follows substituting $(q,a,b,c,t)\rightarrow (q^2,-q,0,-q,q)$ in Heine's ${}_2\phi_1$ transformation in Andrews [p. 19, Corollary 2.3].
\end{proof}

\begin{theorem}\label{thm37}
\begin{equation}\label{eq:P4x-1y1}
\sum_{\lambda\in\mathcal{P}_4}(-1)^{o(\lambda)}q^{|\lambda|}=\sum_{n=0}^{\infty}\frac{(q;q^2)_{n}(-1)^nq^{n^{2}}}{(q^4;q^4)_{n}}=\frac{(q;q^6)_{\infty}(q^5;q^6)_{\infty}(q^6;q^{12})_{\infty}^{2}}{(q^2;q^6)_{\infty}(q^3;q^6)_{\infty}(q^4;q^6)_{\infty}}.   
\end{equation}
\end{theorem}
\begin{proof}Theorem \ref{thm37} follows substituting $(x,y) = (-1,1)$ in \eqref{eq317} above and using the substitution $z=1$ and $q\rightarrow q^3$ in \eqref{eq21}.
\end{proof}

\begin{remark}
The sum of all the even terms and odd terms in \eqref{eq:P4x1y-1} are exactly \eqref{eq:P4eGen} and \eqref{eq:P4oGen} with $(x,y) = (-1,1)$, respectively. By Equation $(53)$ and Equation $(55)$ in Slater's list, the sum of all even terms is also an infinite product, while the sum the all odd terms can be turned to one after multiplying by $q^{-1}$.
\end{remark}

\begin{theorem}\label{thm38}
\begin{equation}\label{eq321}
\sum_{n=0}^{\infty}\frac{q^{n^{2}+n}}{(q^2;q^2)_{n}}=\frac{(-q^2;q^2)_\infty}{(q^4;q^2)_\infty}\left(\frac{(q^2;q^6)_\infty(q^4;q^6)_\infty(q^6;q^6)_\infty}{1-q^2}\right)=(-q^2;q^2)_\infty.   
\end{equation}
\end{theorem}
\begin{proof}Theorem \ref{thm38} follows substituting $(x,y) = (q,q^2)$ in \eqref{eq317} above and using the substitution $z=-q$ and $q\rightarrow q^3$ in \eqref{eq21}.
\end{proof}

\begin{theorem}\label{thm39}
\begin{equation}\label{eq322}
\sum_{n=0}^{\infty}\frac{(-1)^nq^{n^{2}+n}}{(-q^2;q^2)_{n}}=\sum_{n=0}^{\infty}q^{3n^2+n}(1-q^{4n+2})
\end{equation}
\end{theorem}
\begin{proof}
Theorem \ref{thm39} follows substituting $(x,y) = (-q,q^2)$ in \eqref{eq317} above. Also, note that \eqref{eq322} is $q\mapsto q^2$ in \eqref{eq24}.
\end{proof}

\section{Strict overpartitions}\label{sec:SlaterList5}

The Equation~$(8)$ in Slater's list states that
\begin{equation}\label{eq:SlaterList7}
\frac{(q;q)_{\infty}}{(-q;q)_{\infty}}\sum_{n=0}^{\infty}\frac{(-q;q)_{n}q^{\frac{n(n+1)}{2}}}{(q;q)_{n}}=(q;q^4)_{\infty}(q^3;q^4)_{\infty}(q^4;q^4)_{\infty}.  
\end{equation}
This is equivalent to
\begin{equation}\label{eq:SlaterList7Rewrite}
\sum_{n=0}^{\infty}\frac{(-q;q)_{n}q^{\frac{n(n+1)}{2}}}{(q;q)_{n}}=(-q;q)_{\infty}(-q^2;q^2)_{\infty},   
\end{equation}
which is the special case of $a\to1$ in the following theorem.
\begin{theorem}[Lebesgue's identity]
For $|q|<1$,
\begin{equation}\label{eq:qLebesgue}
\sum_{n=0}^{\infty}\frac{(-aq;q)_{n}}{(q;q)_n}q^{\frac{n(n+1)}{2}}=(-aq^2;q^2)_{\infty}(-q;q)_{\infty}.
\end{equation}
\end{theorem}
We wish to give the partition interpretation for a more general series form. Let $\overline{\mathcal{S}}$ be the set of partitions satisfying the following conditions.
\begin{enumerate}
    \item $\lambda$ is a strict partition;
    \item $\lambda_i$ may be overlined if $\lambda_{i}\geq2$ and $\lambda_{i}-\lambda_{i+1}\geq2$.
\end{enumerate}
We call partitions in $\overline{\mathcal{S}}$ strict overpartitions, and let $\overline{\#}(\lambda)$ be the number of overlined parts in $\lambda$. We will prove that
\begin{equation}
\sum_{\lambda\in\overline{\mathcal{S}}}a^{\overline{\#}(\lambda)}z^{\#(\lambda)}q^{|\lambda|}=\sum_{n=0}^{\infty}\frac{(-aq;q)_{n}}{(q;q)_{n}}z^{n}q^{\frac{n(n+1)}{2}},
\end{equation}
so the sum side of \eqref{eq:qLebesgue} becomes a special case of $z\to1$.

It is straightforward that $\overline{\mathcal{S}}$ is an SIP class with modulus $1$, and the basis $\mathcal{B}_{\overline{\mathcal{S}}}$ consists of partitions such that
\begin{enumerate}
    \item $\lambda_i=2$ or $\lambda_i-\lambda_{i+1}=2$ if $\lambda_i$ is overlined;
    \item $\lambda_i=1$ or $\lambda_i-\lambda_{i+1}=1$ otherwise.
\end{enumerate}
\par Let $B_{\overline{S}}(n,h):=B_{\overline{S}}(n,h;a,z,q)$ be the generating function for partitions in $\mathcal{B}_{\overline{\mathcal{S}}}$ with length $n$ and largest part $h$, where $a$, $z$ and $q$ keep track of the number of overlined parts, the length and the weight, respectively.
\begin{theorem}
The initial values for $B_{\overline{S}}(n,h)$ are
\begin{equation}
B_{\overline{S}}(1,h)=\begin{cases}
   zq,  & \text{if $h=1$,} \\
   azq^{2},  & \text{if $h=2$,}\\
   0, & \text{otherwise.}
\end{cases}
\end{equation}
\end{theorem}

\begin{theorem}
For $n\geq1$ and $h\geq2$,
\begin{equation}
B_{\overline{S}}(n,h)=zq^{h}B_{\overline{S}}(n-1,h-1)+azq^{h}B_{\overline{S}}(n-1,h-2).
\end{equation}
\end{theorem}
\begin{proof}
Giving a partition in the basis with length $n$ largest part $h$, we have two cases according to the largest part being overlined or not. In the first case, the second largest part is $h-1$, while in the second case it is $h-2$. This explained the recurrence.
\end{proof}
\begin{theorem}
For any $n\geq1$ and $h\geq1$,
\begin{equation}\label{eq:BSrecurrence}
B_{\overline{S}}(n,n+h)=a^{h}z^nq^{\binom{n+1}{2}+\binom{h+1}{2}}{n\brack h}_{q}.
\end{equation}
\end{theorem}
\begin{proof}
It suffices to check the recurrence relation and the initial values.
\begin{align*}
B_{\overline{S}}(n,n+h)=&a^{h}z^{n}q^{\binom{n+1}{2}+\binom{h+1}{2}}{n\brack h}_{q}\\
&a^{h}z^{n}q^{\binom{n+1}{2}+\binom{h+1}{2}}\left(q^h{n-1\brack h}_{q}+{n-1\brack h-1}_{q}\right)\\
=&a^{h}z^{n}q^{\binom{n+1}{2}+\binom{h+1}{2}+h}{n-1\brack h}_{q}+a^{h}z^{n}q^{\binom{n+1}{2}+\binom{h+1}{2}}{n-1\brack h-1}_{q}\\
&=zq^{n+h}\cdot a^{h}z^{n-1}q^{\binom{n}{2}+\binom{h+1}{2}}{n-1\brack h}_{q}+azq^{n+h}\cdot a^{h-1}z^{n-1}q^{\binom{n}{2}+\binom{h}{2}}{n-1\brack h-1}_{q}\\
=&zq^{n+h}B_{\overline{S}}(n-1,n+h-1)+azq^{n+h}B_{\overline{S}}(n-1,n+h-2),
\end{align*}
so it satisfies~\eqref{eq:BSrecurrence} as desired. It is straightforward to check the value for $n=1$, so we finish the proof.
\end{proof}

\begin{theorem}\label{thm:GFStrictOverpartition}
The generating function for $\overline{\mathcal{S}}$ is \begin{equation}
\sum_{\lambda\in\overline{\mathcal{S}}}a^{\overline{\#}(\lambda)}z^{\#(\lambda)}q^{|\lambda|}=\sum_{n=0}^{\infty}\frac{(-aq;q)_{n}}{(q;q)_{n}}z^nq^{\frac{n(n+1)}{2}}=(-aq)_\infty (-zq)_\infty\sum_{n\ge 0}\dfrac{(-aq)^n}{(q)_n(-zq)_n}.
\end{equation}    
\end{theorem}
\begin{proof}
Since $\overline{\mathcal{S}}$ is an SIP class of modulus $1$,
\begin{align*}
\sum_{\lambda\in\overline{\mathcal{S}}}a^{\overline{\#}(\lambda)}z^{\#(\lambda)}q^{|\lambda|}=&1+\sum_{n=1}^{\infty}\sum_{h=0}^{\infty}\frac{B_{\overline{S}}(n,h)}{(q;q)_{n}}\\
=&1+\sum_{n=1}^{\infty}\frac{z^nq^{\binom{n+1}{2}}}{(q;q)_n}\sum_{h=0}^{\infty}a^{h}q^{\binom{h+1}{2}}{n\brack h}_{q}\\
=&1+\sum_{n=1}^{\infty}\frac{(-aq;q)_{n}z^nq^{\binom{n+1}{2}}}{(q;q)_n}\\
=&\sum_{n=0}^{\infty}\frac{(-aq;q)_{n}z^nq^{\binom{n+1}{2}}}{(q;q)_n}.
\end{align*}
The far right-hand side follows from \eqref{eq25}, so we finish the proof.
\end{proof}
Next we shall establish the sum-product identities related to this generating function. First, we give a new proof for the $q$-Lebesgue identity, together with a combinatorial interpretation.
\begin{theorem}\label{thm:NewQLebesgue}
\begin{equation}
\sum_{\lambda\in\overline{\mathcal{S}}}a^{\overline{\#}(\lambda)}q^{|\lambda|}=\sum_{n=0}^{\infty}\frac{(-aq;q)_{n}}{(q;q)_{n}}q^{\frac{n(n+1)}{2}}=(-aq^2;q^2)_{\infty}(-q;q)_{\infty}.
\end{equation}    
\end{theorem}
\begin{proof}
Letting $t\to -z/a$ and $a\to\infty$ in Cauchy's identity \cite[p. $17$, $(2.2.1)$]{A}, we have
\begin{equation}\label{eq:qBinomThmLimit}
\sum_{n=0}^{\infty}\frac{z^{n}q^{\binom{n}{2}}}{(q;q)_{n}}=(-z;q)_{\infty}.
\end{equation}
Note that, from the second line in the proof of Theorem \ref{thm:GFStrictOverpartition},
\begin{align*}
1+\sum_{n=1}^{\infty}\frac{z^nq^{\binom{n+1}{2}}}{(q;q)_n}\sum_{h=0}^{\infty}a^{h}q^{\binom{h+1}{2}}{n\brack h}_{q}
=&\sum_{n=0}^{\infty}\frac{z^nq^{\binom{n+1}{2}}}{(q;q)_n}\sum_{h=0}^{n}a^{h}q^{\binom{h+1}{2}}\frac{(q;q)_n}{(q;q)_h(q;q)_{n-h}}\\
=&\sum_{h=0}^{\infty}\frac{a^hq^{\binom{h+1}{2}}}{(q;q)_h}\sum_{n=h}^{\infty}\frac{z^{n}q^{\binom{n+1}{2}}}{(q;q)_{n-h}}\\
=&\sum_{h=0}^{\infty}\frac{a^hq^{\binom{h+1}{2}}}{(q;q)_h}\sum_{n=0}^{\infty}\frac{z^{n+h}q^{\binom{n+h+1}{2}}}{(q;q)_{n}}\\
=&\sum_{h=0}^{\infty}\frac{a^hz^hq^{2\binom{h+1}{2}}}{(q;q)_h}\sum_{n=0}^{\infty}\frac{z^{n}q^{\binom{n+1}{2}+nh}}{(q;q)_{n}}\\
\text{(by $z\to zq^{h+1}$ in~\eqref{eq:qBinomThmLimit})}\\
=&\sum_{h=0}^{\infty}\frac{a^hz^hq^{2\binom{h+1}{2}}(-zq^{h+1};q)_{\infty}}{(q;q)_h}\\
=&(-zq;q)_{\infty}\sum_{h=0}^{\infty}\frac{a^hz^hq^{2\binom{h+1}{2}}}{(q;q)_h(-zq;q)_{h}}.
\end{align*}
Let $z\to1$, it is easy to see that
$$\sum_{h=0}^{\infty}\frac{a^hq^{2\binom{h+1}{2}}}{(q;q)_h(-q;q)_{h}}=\sum_{h=0}^{\infty}\frac{a^{h}q^{h^{2}+h}}{(q^2;q^2)_{h}}=(-aq^{2};q^2)_{\infty}.$$
So, we finish the proof.
\end{proof}
Here are some observations. Note that from the proof of Theorem \ref{thm:NewQLebesgue}, we can conclude that
\begin{align}\label{h2}
(-zq;q)_\infty\sum_{h\ge 0}\dfrac{a^h z^h q^{2\binom{h+1}{2}}}{(q;q)_h(-zq;q)_h}=\sum_{n\ge 0}\dfrac{(-aq)_n}{(q)_n}z^n q^{\frac{n(n+1)}{2}}.   
\end{align}
This can also be shown by replacing $t\rightarrow t/(ab)$ and $a, b\rightarrow\infty$ in the second iteration of Heine's transformation in \cite[p. $9$, $(1.2.9)$]{AndrewsBerndt2009}. So, \eqref{h2} and \eqref{eq25} yield 
\begin{align}
(-zq;q)_\infty\sum_{h\ge 0}\dfrac{a^h z^h q^{2\binom{h+1}{2}}}{(q;q)_h(-zq;q)_h}=(-aq)_\infty (-zq)_\infty\sum_{n\ge 0}\dfrac{(-aq)^n}{(q)_n(-zq)_n}.    
\end{align}
Next, we give a new proof for this following identity, which is Equation $(13)$ from Slater's list.
\begin{theorem}
\begin{equation}\label{s13}
\sum_{n=0}^{\infty}\frac{(-q;q)_{n}}{(q;q)_{n}}q^{\frac{n(n-1)}{2}}=\frac{(-q;q)_{\infty}}{(q;q)_{\infty}}\left((q,q^3,q^4;q^4)_{\infty}+(q^2,q^2,q^4;q^4)_{\infty}\right).
\end{equation}
\end{theorem}
\begin{proof} 
Put $a=1$ and $z=q^{-1}$ in Theorem \ref{thm:GFStrictOverpartition} to obtain
\begin{align}\label{f3}
\sum_{n\ge 0}\dfrac{(-q)_n}{(q)_n}q^{\frac{n(n-1)}{2}}&=(-q)_\infty(-1)_\infty\sum_{n\geq 0}\dfrac{(-q)^n}{(q)_n(-1)_n}\notag\\
&=2(-q)_\infty^2\sum_{n\ge 0}\dfrac{(-q)^n(1+q^n)}{2(q)_n(-q)_n}\notag\\
&=(-q)_\infty^2\sum_{n\ge 0}\dfrac{(-q)^n(1+q^n)}{(q^2;q^2)_n}\notag\\
&=(-q)_\infty^2\left(\sum_{n\ge 0}\dfrac{(-1)^nq^n}{(q^2;q^2)_n}+\sum_{n\ge 0}\dfrac{(-1)^nq^{2n}}{(q^2;q^2)_n}\right).
\end{align}
Using Cauchy's identity \cite[p. $17$, $(2.2.1)$]{A} on the two sums in the right-hand side of \eqref{f3}, we obtain
\begin{align*}
\sum_{n\ge 0}\dfrac{(-q)_n}{(q)_n}q^{\frac{n(n-1)}{2}}&=(-q)_\infty^2\left(\dfrac{1}{(-q;q^2)_\infty}+\dfrac{1}{(-q^2;q^2)_\infty}\right)\\
&=(-q)_\infty^2\left(\dfrac{(-q^2;q^2)_\infty}{(-q;q)_\infty}+\dfrac{(-q;q^2)_\infty}{(-q;q)_\infty}\right)\\
&=(-q)_\infty\left((-q^2;q^2)_\infty+(-q;q^2)_\infty\right)
\end{align*}
The product side of \eqref{s13} can be written as
$$\frac{(-q;q)_{\infty}}{(q;q)_{\infty}}\left((q,q^3,q^4;q^4)_{\infty}+(q^2,q^2,q^4;q^4)_{\infty}\right)=(-q;q)_{\infty}\left((-q^2;q^2)_{\infty}+(-q;q^2)_{\infty}\right).$$
So, we finish the proof.
\end{proof}
Finally, we note the following connection with the Rogers-Ramanujan identities. Let
$$P(a,z;q) := {\displaystyle\sum\limits_{n=0}^{\infty}\dfrac{(-aq)_n}{(q)_n}z^nq^\frac{n(n+1)}{2}}.$$
Then, we can show that
\begin{align}
\lim\limits_{a\rightarrow 0}P(1/aq,a;q) = \dfrac{1}{(q,q^4;q^5)_{\infty}}    
\end{align}
and
\begin{align}
\lim\limits_{a\rightarrow 0}P(1/a,a;q) = \dfrac{1}{(q^2,q^3;q^5)_{\infty}}.   
\end{align}
Note that in our three-variable generating function, $a$ and $z$ keep track of the number of overlined parts and length, respectively. So, by letting $(a,z)\to(1/a,a)$, the exponent of $a$ will become the number of non-overlined parts. Finally, by letting $a\to0$, it means that we are counting the number of strict partitions without non-overlined parts. By definition, those are exactly partitions with parts differing by at least $2$, which explains the relation.

\section{Partitions with positional gap condition}\label{sec:SlaterList6}
The following are the Equation $(15)$ and $(19)$ on Slater's list.
\begin{equation}\label{eq:Slater15}
\sum_{n=0}^{\infty}\frac{(-1)^nq^{n(3n-2)}}{(q^4;q^4)_{n}(-q;q^2)_{n}}=\frac{(q,q^4,q^5;q^5)_{\infty}}{(q^2;q^2)_{\infty}},
\end{equation}
\begin{equation}\label{eq:Slater19}
\sum_{n=0}^{\infty}\frac{(-1)^nq^{3n^2}}{(q^4;q^4)_{n}(-q;q^2)_{n}}=\frac{(q^2,q^3,q^5;q^5)_{\infty}}{(q^2;q^2)_{\infty}}.   
\end{equation}
In this section, we consider partitions with positional gap conditions to give a combinatorial interpretation.

\subsection{SIP interpretation for Equation $(15)$}

Let $\mathcal{G}$ be the set of strict partitions such that $\lambda_{2i}-\lambda_{2i+1}$ is even for all $i$, and $\lambda_{2i}$ must be even if it is the smallest part. And for each partition $\lambda$, let $|\lambda_{o}|:=\lambda_1+\lambda_3+\lambda_5+\cdots$ be the sum of all the odd-indexed parts, let $|\lambda_{e}|:=\lambda_2+\lambda_4+\lambda_6+\cdots$ be the sum of all the even-indexed parts. Then $\mathcal{G}$ is an SIP class of modulus $2$, and the basis $\mathcal{B}_{\mathcal{G}}$ consist of partitions such that
\begin{enumerate}
\item $1\leq\lambda_{2i-1}-\lambda_{2i}\leq2$ and $\lambda_{2i}-\lambda_{2i+1}=2$ for all $i$,
\item The smallest part is $1$ or $2$, and it can only be $2$ if it's even-indexed.
\end{enumerate}
Let $B_{G}(n,h):=B_G(n,h;x,y)$ be the generating function for partitions $\lambda$ in $\mathcal{B}_{\mathcal{G}}$ with length $n$ and largest part $h$, where $x$ and $y$ keep track of $|\lambda_{o}|$ and $|\lambda_{e}|$, respectively.
\begin{theorem}
The initial values for $B_G(n,h)$ are
\begin{equation}
B_G(0,h)=\begin{cases}
    1, & \text{if $h=0$}, \\
    0, & \text{otherwise},
\end{cases}    
\end{equation}
and
\begin{equation}
B_G(1,h)=\begin{cases}
    x, & \text{if $h=1$}, \\
    x^2, & \text{if $h=2$},\\
    0, & \text{otherwise}.
\end{cases}
\end{equation}
\end{theorem}

\begin{theorem}
For $n\geq1$,
\begin{equation}\label{eq:BGRec}
B_G(n,h)=x^{h}y^{h-1}B_G(n-2,h-3)+x^{h}y^{h-2}B_G(n-2,h-4).
\end{equation}
\end{theorem}
\begin{proof}
Let $\lambda=(\lambda_1,\lambda_2,\ldots,\lambda_{\#(\lambda)})$ be a partition in $\mathcal{B}_{\mathcal{G}}$ with length $n$ and largest part $h$. Then, $\lambda_{2}$ can be either $h-1$ or $h-2$, and $\lambda_3=\lambda_2-2$ in both cases. So, the weight of $\lambda_1$ and $\lambda_2$ can be either $x^hy^{h-1}$, with $\lambda_3=2h-3$, or $x^hy^{h-2}$, with $\lambda_3=h-4$. So, we have the desired recurrence.     
\end{proof}

\begin{theorem}
\begin{equation}\label{eq:BGOdd}
B_G(2n-1,3n+h-2)=x^{\frac{3n^2-n}{2}+\frac{h^2+h}{2}}y^{\frac{3n^2-3n}{2}+\frac{h^2-h}{2}}{n\brack h}_{xy},
\end{equation}
\begin{equation}\label{eq:BGEven}
B_G(2n,3n+h)=x^{\frac{3n^2+3n}{2}+\frac{h^2+h}{2}}y^{\frac{3n^2+n}{2}+\frac{h^2-h}{2}}{n\brack h}_{xy}.
\end{equation}
\end{theorem}

\begin{proof}
The initial values can be easily checked. So, it suffices to show that both of these expressions satisfy \eqref{eq:BGRec}. For \eqref{eq:BGOdd},
\begin{align*}
&B_G(2n-1,3n+h-2)\\
=&x^{\frac{3n^2-n}{2}+\frac{h^2+h}{2}}y^{\frac{3n^2-3n}{2}+\frac{h^2-h}{2}}{n\brack h}_{xy}\\
=&x^{\frac{3n^2-n}{2}+\frac{h^2+h}{2}}y^{\frac{3n^2-3n}{2}+\frac{h^2-h}{2}}\left(x^hy^{h}{n-1\brack h}_{xy}+{n-1 \brack h-1}\right)\\
=&x^{\frac{3n^2-n}{2}+\frac{h^2+3h}{2}}y^{\frac{3n^2-3n}{2}+\frac{h^2+h}{2}}{n-1\brack h}_{xy}+x^{\frac{3n^2-n}{2}+\frac{h^2+h}{2}}y^{\frac{3n^2-3n}{2}+\frac{h^2-h}{2}}{n-1 \brack h-1}\\
=&x^{3n+h-2}y^{3n+h-3}x^{\frac{3(n-1)^2-(n-1)}{2}+\frac{h^2+h}{2}}y^{\frac{3(n-1)^2-3(n-1)}{2}+\frac{h^2-h}{2}}{n-1\brack h}_{xy}\\
&+x^{3n+h-2}y^{3n+h-4}x^{\frac{3(n-1)^2-(n-1)}{2}+\frac{h^2-h}{2}}y^{\frac{3(n-1)^2-3(n-1)}{2}+\frac{(h-1)^2-(h-1)}{2}}{n-1\brack h-1}_{xy}\\
=&x^{3n+h-2}y^{3n+h-3}B_{G}(2n-3,3n+h-5)+x^{3n+h-2}y^{3n+h-4}B_{G}(2n-3,3n+h-6).
\end{align*}
For \eqref{eq:BGEven},
\begin{align*}
&B_G(2n,3n+h)\\
=&x^{\frac{3n^2+3n}{2}+\frac{h^2+h}{2}}y^{\frac{3n^2+n}{2}+\frac{h^2-h}{2}}{n\brack h}_{xy}\\
=&x^{\frac{3n^2+3n}{2}+\frac{h^2+h}{2}}y^{\frac{3n^2+n}{2}+\frac{h^2-h}{2}}\left(x^hy^{h}{n-1\brack h}_{xy}+{n-1 \brack h-1}\right)\\
=&x^{\frac{3n^2+3n}{2}+\frac{h^2+3h}{2}}y^{\frac{3n^2+n}{2}+\frac{h^2+h}{2}}{n-1\brack h}_{xy}+x^{\frac{3n^2+3n}{2}+\frac{h^2+h}{2}}y^{\frac{3n^2+n}{2}+\frac{h^2-h}{2}}{n-1 \brack h-1}\\
=&x^{3n+h}y^{3n+h-1}x^{\frac{3(n-1)^2-3(n-1)}{2}+\frac{h^2+h}{2}}y^{\frac{3(n-1)^2+(n-1)}{2}+\frac{h^2-h}{2}}{n-1\brack h}_{xy}\\
&+x^{3n+h}y^{3n+h-2}x^{\frac{3(n-1)^2-3(n-1)}{2}+\frac{h^2-h}{2}}y^{\frac{3(n-1)^2+(n-1)}{2}+\frac{(h-1)^2-(h-1)}{2}}{n-1\brack h-1}_{xy}\\
=&x^{3n+h}y^{3n+h-1}B_{G}(2n-2,3n+h-3)+x^{3n+h}y^{3n+h-2}B_{G}(2n-2,3n+h-4).
\end{align*}
So, we finish the proof.
\end{proof}

\begin{theorem}\label{thm:GFGxy}
\begin{equation}\label{eq:GFGxy}
\sum_{\lambda\in\mathcal{G}}x^{|\lambda_{o}|}y^{|\lambda_{e}|}=\sum_{n=0}^{\infty}\frac{x^{\frac{3n^2-n}{2}}y^{\frac{3n^2-3n}{2}}}{(-x;xy)_n(x^2y^2;x^2y^2)_{n}}.
\end{equation}
\end{theorem}

\begin{proof}
Since we consider the weight of even-indexed parts and odd-indexed parts separately, as an SIP of modulus $2$, the generating function for $\mathcal{G}$ is
\begin{align*}
\sum_{\lambda\in\mathcal{G}}x^{|\lambda_{o}|}y^{|\lambda_{e}|}=&\sum_{n=1}^{\infty}\sum_{h=0}^{\infty}\frac{B_G(2n-1,3n+h-2)}{(x^2;x^2y^2)_n(x^2y^2;x^2y^2)_{n-1}}+\sum_{n=0}^{\infty}\sum_{h=0}^{\infty}\frac{B_G(2n,3n+h)}{(x^2;x^2y^2)_n(x^2y^2;x^2y^2)_{n}}\\
=&\sum_{n=1}^{\infty}\frac{x^{\frac{3n^2-n}{2}}y^{\frac{3n^2-3n}{2}}}{(x^2;x^2y^2)_n(x^2y^2;x^2y^2)_{n-1}}\sum_{h=0}^{\infty}x^{\frac{h^2+h}{2}}y^{\frac{h^2-h}{2}}{n\brack h}_{xy}\\
&+\sum_{n=0}^{\infty}\frac{x^{\frac{3n^2+3n}{2}}y^{\frac{3n^2+n}{2}}}{(x^2;x^2y^2)_n(x^2y^2;x^2y^2)_{n}}\sum_{h=0}^{\infty}x^{\frac{h^2+h}{2}}y^{\frac{h^2-h}{2}}{n\brack h}_{xy}\\
&\text{(by~\eqref{eq22} with $z\to x$ and $q\to xy$)}\\
=&\sum_{n=1}^{\infty}\frac{(-x;xy)_{n}x^{\frac{3n^2-n}{2}}y^{\frac{3n^2-3n}{2}}}{(x^2;x^2y^2)_n(x^2y^2;x^2y^2)_{n-1}}+\sum_{n=0}^{\infty}\frac{(-x;xy)_{n}x^{\frac{3n^2+3n}{2}}y^{\frac{3n^2+n}{2}}}{(x^2;x^2y^2)_n(x^2y^2;x^2y^2)_{n}}\\
=&1+\sum_{n=1}^{\infty}\frac{(-x;xy)_{n}x^{\frac{3n^2-n}{2}}y^{\frac{3n^2-3n}{2}}}{(x^2;x^2y^2)_n(x^2y^2;x^2y^2)_{n-1}}\left(1+\frac{x^{2n}y^{2n}}{1-x^{2n}y^{2n}}\right)\\
=&\sum_{n=0}^{\infty}\frac{(-x;xy)_{n}x^{\frac{3n^2-n}{2}}y^{\frac{3n^2-3n}{2}}}{(x^2;x^2y^2)_n(x^2y^2;x^2y^2)_{n}}\\
=&\sum_{n=0}^{\infty}\frac{x^{\frac{3n^2-n}{2}}y^{\frac{3n^2-3n}{2}}}{(x;xy)_n(x^2y^2;x^2y^2)_{n}}.
\end{align*}
So we finish the proof.
\end{proof}
Given a partition $\lambda$, let $a(\lambda):=|\lambda_{o}|-|\lambda_{e}|$ be its alternating sum. Let $x\to zq$ and $y\to q/z$ in~\eqref{eq:GFGxy}, we have the following.
\begin{theorem}
\begin{equation}\label{eq:GFGzq}
\sum_{\lambda\in\mathcal{G}}z^{a(\lambda)}q^{|\lambda|}=\sum_{n=0}^{\infty}\frac{z^{n}q^{3n^2-2n}}{(zq;q^2)_n(q^4;q^4)_{n}}.   
\end{equation}
\end{theorem}
Note that the sum side of~\eqref{eq:Slater15} is the special case of~\eqref{eq:GFGzq} with $z\to-1$. To find a product side, it is routine to change the order of summation. If we set $x\to zq$ and $y\to q/z$ in the proof of Theorem~\ref{thm:GFGxy},
\begin{align*}
\sum_{\lambda\in\mathcal{G}}z^{a(\lambda)}q^{|\lambda|}=&\sum_{n=1}^{\infty}\frac{z^{n}q^{3n^2-2n}}{(z^2q^2;q^4)_n(q^4;q^4)_{n-1}}\sum_{h=0}^{\infty}z^{h}q^{h^2}{n \brack h}_{q^2}+\sum_{n=0}^{\infty}\frac{z^{n}q^{3n^2+2n}}{(z^2q^2;q^4)_n(q^4;q^4)_n}\sum_{h=0}^{\infty}z^{h}q^{h^2}{n \brack h}_{q^2}\\
=&\sum_{h=0}^{\infty}z^{h}q^{h^2}\left(\sum_{n=1}^{\infty}\frac{z^{n}q^{3n^2-2n}}{(z^2q^2;q^4)_n(q^4;q^4)_{n-1}}{n \brack h}_{q^2}+\sum_{n=0}^{\infty}\frac{z^{n}q^{3n^2+2n}}{(z^2q^2;q^4)_n(q^4;q^4)_n}{n \brack h}_{q^2}\right)\\
=&\sum_{h=0}^{\infty}z^{h}q^{h^2}\left(1+\sum_{n=1}^{\infty}\frac{z^{n}q^{3n^2-2n}}{(z^2q^2;q^4)_n(q^4;q^4)_{n-1}}{n \brack h}_{q^2}+\sum_{n=1}^{\infty}\frac{z^{n}q^{3n^2+2n}}{(z^2q^2;q^4)_n(q^4;q^4)_n}{n \brack h}_{q^2}\right)\\
=&\sum_{h=0}^{\infty}z^{h}q^{h^2}\left(1+\sum_{n=1}^{\infty}\frac{z^{n}q^{3n^2-2n}}{(z^2q^2;q^4)_n(q^4;q^4)_{n-1}}\left(1+\frac{q^{4n}}{1-q^{4n}}\right){n \brack h}_{q^2}\right)\\
=&\sum_{h=0}^{\infty}z^{h}q^{h^2}\left(1+\sum_{n=1}^{\infty}\frac{z^{n}q^{3n^2-2n}}{(z^2q^2;q^4)_n(q^4;q^4)_{n}}{n \brack h}_{q^2}\right)\\
=&\sum_{h=0}^{\infty}z^{h}q^{h^2}\sum_{n=0}^{\infty}\frac{z^{n}q^{3n^2-2n}}{(z^2q^2;q^4)_n(q^4;q^4)_{n}}{n \brack h}_{q^2}\\
=&\sum_{h=0}^{\infty}\frac{z^{h}q^{h^2}}{(q^2;q^2)_h}\sum_{n=h}^{\infty}\frac{z^{n}q^{3n^2-2n}}{(z^2q^2;q^4)_n(-q^2;q^2)_{n}(q^2;q^2)_{n-h}}\\
=&\sum_{h=0}^{\infty}\frac{z^{2h}q^{4h^2-2h}}{(q^2;q^2)_h(z^2q^2;q^4)_h(-q^2;q^2)_h}\sum_{n=0}^{\infty}\frac{z^{n}q^{3n^2+6nh-2n}}{(z^2q^{4h+2};q^4)_n(-q^{2h+2};q^2)_{n}(q^2;q^2)_{n}}\\
=&\sum_{h=0}^{\infty}\frac{z^{2h}q^{4h^2-2h}}{(q^4;q^4)_h(z^2q^2;q^4)_h}\sum_{n=0}^{\infty}\frac{z^{n}q^{3n^2+6nh-2n}}{(zq^{2h+1};q^2)_n(-zq^{2h+1};q^2)_{n}(-q^{2h+2};q^2)_{n}(q^2;q^2)_{n}}.
\end{align*}
However, it seems that the inner sum is still hard to handle, and even if we set $z\to-1$, as the special case of \eqref{eq:Slater19}, we couldn't reach the product side. So, for this identity, we only provided a partition interpretation for the series.

\subsection{SIP interpretation for Equation $(19)$}

Let $\mathcal{G}'$ be the set of partitions such that
\begin{enumerate}
    \item $\lambda_{2i-1}-\lambda_{2i}\geq3$ and $\lambda_{2i}-\lambda_{2i+1}$ is even for all $i$;
    \item the smallest part must be at least $3$ if it is odd-indexed, and must be even if it is even-indexed.
\end{enumerate}
Then $\mathcal{G}'$ is an SIP class with modulus $2$, and the basis $\mathcal{B}_{\mathcal{G}'}$ consists of partitions such that
\begin{enumerate}
    \item $3\leq\lambda_{2i-1}-\lambda_{2i}\leq4$ and $\lambda_{2i}-\lambda_{2i+1}=0$ for all $i$;
    \item the smallest part can be either $3$ or $4$ if it is odd-indexed, and can be only $2$ if it is even-indexed.
\end{enumerate}
Let $B_{G'}(n,h):=B_{G'}(n,h;x,y)$ be the generating function of partitions in $\mathcal{B}_{\mathcal{G}'}$ with length $n$ and largest part $h$, where $x$ and $y$ keep track of $|\lambda_{o}|$ and $|\lambda_{e}|$, respectively.

\begin{theorem}
The initial values for $B_{G'}(n,h)$ are
\begin{equation}
B_{G'}(0,h)=\begin{cases}
    1 & \text{if $h=0$}, \\
    0 & \text{otherwise},
\end{cases}    
\end{equation}
and
\begin{equation}
B_{G'}(1,h)=\begin{cases}
    x^3 & \text{if $h=3$}, \\
    x^4 & \text{if $h=4$},\\
    0 & \text{otherwise}.
\end{cases}
\end{equation}
\end{theorem}

\begin{theorem}
For $n\geq1$,
\begin{equation}\label{eq:BG'Rec}
B_{G'}(n,h)=x^{h}y^{h-3}B_G(n-2,h-3)+x^{h}y^{h-4}B_G(n-2,h-4).
\end{equation}
\end{theorem}
\begin{proof}
Giving a partition $\lambda=(\lambda_1,\lambda_2,\ldots,\lambda_{\#(\lambda)})$ in $\mathcal{B}_{\mathcal{G}'}$ with length $n$ and $\lambda_1=h$, then either $\lambda_2=\lambda_3=h-3$ or $\lambda_2=\lambda_3=h-4$. So, the weight of $\lambda_1$ and $\lambda_2$ is either $x^hy^{h-3}$ or $x^hy^{h-3}$. By deleting $\lambda_1$ and $\lambda_2$, we get a partition in the basis with length $n-2$, so we have the recurrence.   
\end{proof}

\begin{theorem}
\begin{equation}\label{eq:BG'Odd}
B_{G'}(2n-1,3n+h)=x^{\frac{3n^2+3n}{2}+\frac{h^2+h}{2}}y^{\frac{3n^2-3n}{2}+\frac{h^2-h}{2}}{n\brack h}_{xy},
\end{equation}
\begin{equation}\label{eq:BG'Even}
B_{G'}(2n,3n+h+2)=x^{\frac{3n^2+7n}{2}+\frac{h^2+h}{2}}y^{\frac{3n^2+n}{2}+\frac{h^2-h}{2}}{n\brack h}_{xy}.
\end{equation}
\end{theorem}

\begin{proof}
We first show that these expressions satisfy the recurrence in \eqref{eq:BG'Rec}. For \eqref{eq:B4OddRec},
\begin{align*}
&B_{G'}(2n-1,3n+h)\\
=&x^{\frac{3n^2+3n}{2}+\frac{h^2+h}{2}}y^{\frac{3n^2-3n}{2}+\frac{h^2-h}{2}}{n\brack h}_{xy}\\
=&x^{\frac{3n^2+3n}{2}+\frac{h^2+h}{2}}y^{\frac{3n^2-3n}{2}+\frac{h^2-h}{2}}\left(x^hy^h{n-1\brack h}_{xy}+{n-1\brack h-1}_{xy}\right)\\
=&x^{\frac{3n^2+3n}{2}+\frac{h^2+3h}{2}}y^{\frac{3n^2-3n}{2}+\frac{h^2+h}{2}}{n-1\brack h}_{xy}+x^{\frac{3n^2+3n}{2}+\frac{h^2+h}{2}}y^{\frac{3n^2-3n}{2}+\frac{h^2-h}{2}}{n-1\brack h-1}_{xy}\\
=&x^{3n+h}y^{3n+h-3}x^{\frac{3n^2-3n}{2}+\frac{h^2+h}{2}}y^{\frac{3n^2-9n+6}{2}+\frac{h^2-h}{2}}{n-1\brack h}_{xy}\\
&+x^{3n+h}y^{3n+h-4}x^{\frac{3n^2-3n}{2}+\frac{h^2-h}{2}}y^{\frac{3n^2-9n+6}{2}+\frac{h^2-3h+2}{2}}{n-1\brack h-1}_{xy}\\
=&x^{3n+h}y^{3n+h-3}B_{G'}(2n-3,3n+h-3)+x^{3n+h}y^{3n+h-4}B_{G'}(2n-3,3n+h-4).
\end{align*}
For \eqref{eq:BG'Even},
\begin{align*}
&B_{G'}(2n,3n+h+2)\\
=&x^{\frac{3n^2+7n}{2}+\frac{h^2+h}{2}}y^{\frac{3n^2+n}{2}+\frac{h^2-h}{2}}{n\brack h}_{xy}\\
=&x^{\frac{3n^2+7n}{2}+\frac{h^2+h}{2}}y^{\frac{3n^2+n}{2}+\frac{h^2-h}{2}}\left(x^hy^h{n-1\brack h}_{xy}+{n-1\brack h-1}_{xy}\right)\\
=&x^{\frac{3n^2+7n}{2}+\frac{h^2+3h}{2}}y^{\frac{3n^2+n}{2}+\frac{h^2+h}{2}}{n-1\brack h}_{xy}+x^{\frac{3n^2+7n}{2}+\frac{h^2+h}{2}}y^{\frac{3n^2+n}{2}+\frac{h^2-h}{2}}{n-1\brack h-1}_{xy}\\
=&x^{3n+h+2}y^{3n+h-1}x^{\frac{3n^2+n-4}{2}+\frac{h^2+h}{2}}y^{\frac{3n^2-5n+2}{2}+\frac{h^2-h}{2}}{n-1\brack h}_{xy}\\
&+x^{3n+h+2}y^{3n+h-2}x^{\frac{3n^2+n-4}{2}+\frac{h^2-h}{2}}y^{\frac{3n^2-5n+2}{2}+\frac{h^2-3h+2}{2}}{n-1\brack h-1}_{xy}\\
=&x^{3n+h+2}y^{3n+h-1}B_{G'}(2n-2,3n+h-1)+x^{3n+h+2}y^{3n+h-2}B_{G'}(2n-2,3n+h-2).
\end{align*}
It is straightforward to check the initial values, so we finish the proof.
\end{proof}

\begin{theorem}\label{thm:GFG'xy}
\begin{equation}\label{eq:GFG'xy}
\sum_{\lambda\in\mathcal{G}'}x^{|\lambda_{o}|}y^{|\lambda_{e}|}=\sum_{n=0}^{\infty}\frac{x^{\frac{3n^2+3n}{2}}y^{\frac{3n^2-3n}{2}}}{(-x;xy)_n(x^2y^2;x^2y^2)_{n}}.
\end{equation}
\end{theorem}

\begin{proof}
Since we consider the weight of even-indexed parts and odd-indexed parts separately, as an SIP of modulus $2$, the generating function for $\mathcal{G}'$ is
\begin{align*}
\sum_{\lambda\in\mathcal{G}'}x^{|\lambda_{o}|}y^{|\lambda_{e}|}=&\sum_{n=1}^{\infty}\sum_{h=0}^{\infty}\frac{B_{G'}(2n-1,3n+h)}{(x^2;x^2y^2)_n(x^2y^2;x^2y^2)_{n-1}}+\sum_{n=0}^{\infty}\sum_{h=0}^{\infty}\frac{B_{G'}(2n,3n+h+2)}{(x^2;x^2y^2)_n(x^2y^2;x^2y^2)_{n}}\\
=&\sum_{n=1}^{\infty}\frac{x^{\frac{3n^2+3n}{2}}y^{\frac{3n^2-3n}{2}}}{(x^2;x^2y^2)_n(x^2y^2;x^2y^2)_{n-1}}\sum_{h=0}^{\infty}x^{\frac{h^2+h}{2}}y^{\frac{h^2-h}{2}}{n\brack h}_{xy}\\
&+\sum_{n=0}^{\infty}\frac{x^{\frac{3n^2+7n}{2}}y^{\frac{3n^2+n}{2}}}{(x^2;x^2y^2)_n(x^2y^2;x^2y^2)_{n}}\sum_{h=0}^{\infty}x^{\frac{h^2+h}{2}}y^{\frac{h^2-h}{2}}{n\brack h}_{xy}\\
&\text{(by~\eqref{eq21} with $z\to x$ and $q\to xy$)}\\
=&\sum_{n=1}^{\infty}\frac{(-x;xy)_{n}x^{\frac{3n^2+3n}{2}}y^{\frac{3n^2-3n}{2}}}{(x^2;x^2y^2)_n(x^2y^2;x^2y^2)_{n-1}}+\sum_{n=0}^{\infty}\frac{(-x;xy)_{n}x^{\frac{3n^2+7n}{2}}y^{\frac{3n^2+n}{2}}}{(x^2;x^2y^2)_n(x^2y^2;x^2y^2)_{n}}\\
=&1+\sum_{n=1}^{\infty}\frac{(-x;xy)_{n}x^{\frac{3n^2+3n}{2}}y^{\frac{3n^2-3n}{2}}}{(x^2;x^2y^2)_n(x^2y^2;x^2y^2)_{n-1}}\left(1+\frac{x^{2n}y^{2n}}{1-x^{2n}y^{2n}}\right)\\
=&\sum_{n=0}^{\infty}\frac{(-x;xy)_{n}x^{\frac{3n^2+3n}{2}}y^{\frac{3n^2-3n}{2}}}{(x^2;x^2y^2)_n(x^2y^2;x^2y^2)_{n}}\\
=&\sum_{n=0}^{\infty}\frac{x^{\frac{3n^2+3n}{2}}y^{\frac{3n^2-3n}{2}}}{(x;xy)_n(x^2y^2;x^2y^2)_{n}}.
\end{align*}
So we finish the proof.
\end{proof}

Let $x\to zq$ and $y\to q/z$ in~\eqref{eq:GFGxy}, we have the following.
\begin{theorem}
\begin{equation}\label{eq:GFG'zq}
\sum_{\lambda\in\mathcal{G'}}z^{a(\lambda)}q^{|\lambda|}=\sum_{n=0}^{\infty}\frac{z^{n}q^{3n^2}}{(zq;q^2)_n(q^4;q^4)_{n}}.
\end{equation}
\end{theorem}
\begin{remark}\label{rmk2}
Note that the sum side of~\eqref{eq:Slater19} is the special case of~\eqref{eq:GFG'zq} with $z = -1$. Applying \eqref{eq26} with $q\rightarrow q^2$, $a, b\rightarrow\infty$, $d=zq, e=-q^2$ (or $d=-q^2, e=zq$) yields
\begin{align*}
\sum_{n\ge 0}\dfrac{z^n q^{3n^2}}{(zq;q^2)_n(q^4;q^4)_n}=\dfrac{1}{(-q^2;q^2)_\infty}\sum_{n\ge 0}\dfrac{q^{n^2+n}}{(zq;q^2)_n(q^2;q^2)_n}=\dfrac{1}{(zq;q^2)_\infty}\sum_{n\ge 0}\dfrac{(-z)^n q^{n^2}}{(q^4;q^4)_n}.\numberthis\label{eq517}    
\end{align*}
For $z = -1$, the sum on farthest right becomes a product due to Rogers' following identity \cite{Rogers1894}.
\begin{align*}
\sum\limits_{n\ge 0}\dfrac{q^{n^2}}{(q^4;q^4)_n} = \dfrac{1}{(-q^2;q^2)_{\infty}(q,q^4;q^5)_{\infty}}.\numberthis\label{eq518}    
\end{align*}
Hence, from \eqref{eq517} and \eqref{eq518}, we have
\begin{align*}
\sum\limits_{n\ge 0}\dfrac{(-1)^nq^{3n^2}}{(-q;q^2)_n(q^4;q^4)_n} &= \dfrac{1}{(-q,-q^2;q^2)_{\infty}(q,q^4;q^5)_{\infty}}\\
&= \dfrac{1}{(-q;q)_{\infty}(q,q^4;q^5)_{\infty}}\\
&= \dfrac{(q;q)_{\infty}}{(q^2;q^2)_{\infty}(q,q^4;q^5)_{\infty}}\\
&= \dfrac{(q^2,q^3,q^5;q^5)_{\infty}}{(q^2;q^2)_{\infty}},
\end{align*}
which gives us \eqref{eq:Slater19}.\\
\end{remark}

\section{A Comment on the Combinatorics}\label{sec:SlaterList7}
Some of the identities in Slater's list can be treated in a purely combinatorial manner. For instance, the Equation~$(5)$ in Slater's list states that
\begin{equation}\label{eq:SlaterList5}
\sum_{n=0}^{\infty}\frac{(-1)^{n}q^{n(2n+1)}}{(q^2;q^2)_{n}(-q;q^{2})_{n+1}}=(q;q^2)_{\infty}(-q^2;q^2)_{\infty}.    
\end{equation}
Consider the following identity,
\begin{equation}\label{eq:PartitionOdd}
\sum_{n=0}^{\infty}\frac{z^nq^{n(2n+1)}}{(q^2;q^2)_{n}(zq;q^2)_{n+1}}=\frac{1}{(zq;q^2)_{\infty}}.
\end{equation}
Letting $z\to-1$ in~\eqref{eq:PartitionOdd}, the right-hand side is
$$\frac{1}{(-q;q^2)_{\infty}}=\frac{(-q^2;q^2)_{\infty}}{(-q;q)_{\infty}}=(q;q^2)_{\infty}(-q^2;q^2)_{\infty},$$
which would imply~\eqref{eq:SlaterList5} as a corollary. Here we present a bijective proof for~\eqref{eq:PartitionOdd}. Let $\mathcal{O}$ be the set of partitions into odd parts. It is clear that
$$\sum_{\lambda\in\mathcal{O}}z^{\#(\lambda)}q^{|\lambda|}=\frac{1}{(zq;q^2)_{\infty}},$$
so it suffices to show that the left-hand side of~\eqref{eq:PartitionOdd} generates the same objects.
\begin{proof}
The partitions into odd parts can be represented as the following mod $2$ Ferrers diagram.
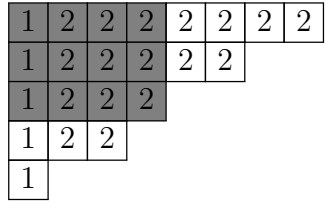
\begin{center}
\begin{ytableau}
*(gray)1 & *(gray)2 & *(gray)2 & *(gray)2 & 2 & 2 & 2 & 2\\
*(gray)1 & *(gray)2 & *(gray)2 & *(gray)2 & 2 & 2\\
*(gray)1 & *(gray)2 & *(gray)2 & *(gray)2\\
1 & 2 & 2\\
1
\end{ytableau}
\end{center}
In each row, the first box is filled by a $1$, while the rest are filled by $2$'s. We find the largest integer $n$ such that an $n\times(n+1)$ rectangle is contained in the diagram, like the gray area. This rectangle is generated by $z^{n}q^{n(2n+1)}$, where $z$ keeps track of the length. The portion to the right is clearly $1/(q^2;q^2)_{n}$, while the portion below is $1/(zq;q^2)_{n+1}$. For each diagram, the value of $n$ is unique. So, by summing up over all non-negative integers $n$,
$$\sum_{\lambda\in\mathcal{O}}z^{\#(\lambda)}q^{|\lambda|}=\sum_{n=0}^{\infty}\frac{z^nq^{n(2n+1)}}{(q^2;q^2)_{n}(zq;q^2)_{n+1}}.$$
We finish the proof.
\end{proof}

For most of the identities on Slater's list, such a direct combinatorial proof does not seem attainable. However, finding the partition interpretations for those identities would be a good starting point. The product side is usually ease to be translated into partitions with certain congruence properties. For example, in Section~$3$, we treated Equation~$(4)$ from Slater's list, namely,
\begin{equation}
\sum_{n=0}^{\infty}\frac{(-1)^{n}(-q;q^2)_{n}q^{n^2}}{(q^4;q^4)_{n}}=(q;q^2)_{\infty}(q^2;q^4)_{\infty}. 
\end{equation}
Let $\mathcal{S}_4$ be the set of strict partitions without multiples of $4$, then
$$\sum_{\lambda\in\mathcal{S}_4}z^{\#(\lambda)}q^{|\lambda|}=(-zq;q^2)_{\infty}(-zq^2;q^4)_{\infty}.$$
The product side of Equation~$(4)$ is just the special case of $z\to-1$. By Theorem~\ref{thm:P4Gen}, Equation~$(4)$ is equivalent to
$$\sum_{\lambda\in\mathcal{P}_4}(-1)^{\#(\lambda)}q^{|\lambda|}=\sum_{\lambda\in\mathcal{S}_4}(-1)^{\#(\lambda)}q^{|\lambda|}.$$
This can be rephrased as the following partition identity.
\begin{theorem}
Let $\mathcal{P}_{4,e}(n)$ and $\mathcal{P}_{4,o}(n)$ be the set of partitions of $n$ in $\mathcal{P}_4$ with even length and odd length, respectively. And let $\mathcal{S}_{4,e}(n)$ and $\mathcal{S}_{4,o}(n)$ be the set of partitions of $n$ in $\mathcal{S}_4$ with even length and odd length, respectively. For any $n\geq0$,
$$\mathcal{P}_{4,e}(n)-\mathcal{P}_{4,o}(n)=\mathcal{S}_{4,e}(n)-\mathcal{S}_{4,o}(n).$$
\end{theorem}
One can construct involutions on $\mathcal{P}_4$ and $\mathcal{S}_4$ to characterize the difference on both sides, and then try to find a bijection between them. This would provide a combinatorial proof for Equation~$(4)$. The same process now can be applied to all the identities we treated in this project. Using the partition sets we constructed through SIP classes, an equivalent partition identity can be presented for each of the identities. Thus, this project could lead to a further study of Slater's list from a combinatorial point of view.\\

\section*{Acknowledgments}
Aritram Dhar thanks George E. Andrews for hosting him as a visiting research scholar at the Pennsylvania State University from May 7--June 25, 2025, during which this project began following discussions with Runqiao Li. Ankush Goswami also thanks George E. Andrews for the warm hospitality during his one-month visit to Penn State (July--August 2023), where he was introduced to a related project on SIP.

\end{document}